\documentclass{amsart}
\usepackage{todonotes}
\usepackage{amsmath,amssymb,dsfont}
\usepackage{graphicx,color}
\usepackage{pstricks, pst-plot, pst-node, pst-grad, pst-math}
\usepackage{pst-plot}
\newtheorem{theorem}{Theorem}[section]
\newtheorem{lemma}[theorem]{Lemma}
\newtheorem{proposition}[theorem]{Proposition}

\theoremstyle{definition}

\numberwithin{equation}{section}

\newcommand\lm{\lambda}
\newcommand\Lm{\Lambda}

\newcommand{\rd}{{\,\rm d}}
\newcommand{\e}{{\rm e}}

\newcommand{\N}{{\mathbb N}}

\newcommand{\R}{{\mathbb R}}
\newcommand{\C}{{\mathbb C}}

\newcommand\beq{\begin{equation}}
\newcommand\eeq{\end{equation}}

\newcommand{\dist}{\mathrm{dist}}
\newcommand\re{\mathrm{Re}}
\newcommand\im{\mathrm{Im}}
\newcommand\I{\mathrm{i}}
\newcommand{\beqnt}{\begin{equation*}}
\newcommand{\eeqnt}{\end{equation*}}
\newcommand{\set}[2]{\{#1 : #2 \}}

\newcommand{\sgn}{\operatorname{sgn}}

\begin{document}

\title[Landau Hamiltonian with $L^p$ potentials]{Sharp spectral estimates for the perturbed Landau Hamiltonian with $L^p$ potentials}

\author{Jean-Claude Cuenin}
\address{Mathematisches Institut, Ludwig-Maximilians-Universit\"at M\"unchen, 80333 Munich, Germany}
\email{cuenin@math.lmu.de}

\begin{abstract}
We establish a sharp estimate on the size of the spectral clusters of the Landau Hamiltonian with $L^p$ potentials in two dimensions as the cluster index tends to infinity. In three dimensions, we prove a new limiting absorption principle as well as 
a unique continuation theorem. The results generalize to higher dimensions.  
\end{abstract}

\maketitle

\section{Introduction}
Consider the Hamiltonian $H_{0,\perp}$ of a particle confined to a two-dimensional surface in a constant magnetic field $B=(0,0,1)$, perpendicular to the surface. In the symmetric gauge, the so-called Landau Hamiltonian is given by
\beq\label{eq. unperturbed Landau Hamiltonian 2D}
H_{0,\perp}=\left(-\I\frac{\partial}{\partial x}+\frac{y}{2}\right)^2+\left(-\I\frac{\partial}{\partial y}-\frac{x}{2}\right)^2,\quad (x,y)\in\R^2.
\eeq
The spectrum of $H_{0,\perp}$ is purely discrete, 
\[
\sigma(H_{0,\perp})=\set{\lm_k=2k+1}{k\in\N_0},
\]
and each eigenvalue $\lm_k$ (called the $k$-th Landau level) is of infinite multiplicity. We are interested in the spectrum of the perturbed Landau Hamiltonian
\beq\label{eq. perturbed Landau Hamiltonian}
H=H_{0,\perp}+V
\eeq
where $V$ is a (complex-valued) potential. In particular, we derive sharp bounds on the location of the spectrum of $H$ lying in the $k_0$-th spectral cluster
\begin{align}\label{def. Lambdak0}
\Lm_{k_0}=\set{z\in\C}{|\lm_{k_0}-\re z|\leq 1}
\end{align}
as $k_0\to\infty$. If $V\in L^{\infty}(\R^2)$, then by standard perturbation theory
\[
\sigma(H)\subset \set{z\in\C}{\min_{k\in\N_0}|\lm_k-z|\leq \|V\|_{\infty}}.
\]
On the other hand, if $V$ is a real-valued smooth function of compact support, then it was shown in \cite{MR2097614} that
\begin{align}\label{eq. cluster asymptotics literature}
\sigma(H)\cap\Lambda_{k_0}\subset [\lm_{k_0}-C\lm_{k_0}^{-1/2},\lm_{k_0}+C\lm_{k_0}^{-1/2}]
\end{align}
for $k_0$ sufficiently large. Here, $C>0$ is a constant depending on $\|V\|_{\infty}$ and on the diameter of the support of $V$. This result was also proven in \cite{MR3053767} by different techniques and for more general potentials, namely for continuous $V$ such that
\begin{align*}
\|V\|_{X_{\rho}}:=\sup_{x\in\R^2}(1+x^2)^{\rho/2}|V(x)|<\infty
\end{align*}
for some $\rho>1.$ The $O(\lm_{k_0}^{-1/2})$-decay for the size of the $k_0$-th cluster is sharp, i.e.\ the eigenvalue clusters have size $\geq c\lm_{k_0}^{-1/2}$ for some $c>0$ unless $V\equiv 0$. Moreover, the constant $C$ in \eqref{eq. cluster asymptotics literature} depends only on $\|V\|_{X_{\rho}}$. For $\rho<1$, the optimal decay rate of the spectral cluster size is $\mathcal{O}(\lm_{k_0}^{-\rho/2})$ \cite{MR3232579}. 

The aim of this note is twofold: First we obtain estimates on the size of the spectral clusters that depend only on an $L^p$-norm of $V$.
The result (Theorem \ref{thm spectral clusters}) holds more generally in all even dimensions. Second, when the dimension is odd, we prove an $L^p$ limiting absorption principle (Theorem~\ref{LAP small perturbation}) and a unique continuation theorem (Theorem \ref{thm unique continuation}).

We will use the standard notation $a\lesssim b$ if there exists a non-negative constant $C$ such that $a\leq Cb$. If we want to emphasize the dependence of the constant on some parameter $J$, we write $a\lesssim_J b$. If $a\lesssim b\lesssim a$, we write $a\approx b$. 

\section{Spectral cluster estimates in even dimensions}
In $d=2n$ dimensions, we consider the generalization of \eqref{eq. unperturbed Landau Hamiltonian even d}, i.e.
\beq\label{eq. unperturbed Landau Hamiltonian even d}
H_{0,\perp}=\sum_{j=1}^n\left(-\I\frac{\partial}{\partial x_j}+\frac{y_j}{2}\right)^2+\left(-\I\frac{\partial}{\partial y_j}-\frac{x_j}{2}\right)^2,\quad (x,y)\in\R^{2n}.
\eeq
Its eigenvalues are given by $\lambda_k:=2k+n$, $k\in\N$.

Assume that $V\in L^{r}(\R^d)$, with $r\in [d/2,\infty]$, is a (possibly complex-valued) potential.
The perturbed operator (defined in the sense of sectorial forms) then satisfies the spectral estimate \cite[Theorem 5.1]{2015arXiv151000066C}
\begin{align*}
\sigma(H)\subset\left\{z\in\C:| \im z|^{1-\frac{d}{2r}}\lesssim 1+ \|V\|_{L^r(\R^d)}\right\}.
\end{align*}
In particular, if $r>d/2$, this implies that $|\im z|\lesssim 1$ for all $z\in\sigma(H)$.
Our main result is the following refinement.

\begin{theorem}[Shrinking spectral clusters]\label{thm spectral clusters} 
Let $d\in 2\N$, and let $V\in L^r(\R^d)$ with $r\in [d/2,\infty]$. Then there exists $K\in\N$ such that for all $k_0\geq K$ we have 
\begin{align}\label{shrinking spectral cluster estimate}
\sigma(H)\cap\Lambda_{k_0}\subset\set{z\in\C}{\delta(z)\lesssim \|V\|_{L^r(\R^d)}\lambda_{k_0}^{\nu(r)}},
\end{align}
where $\delta(z):=\dist(z,\sigma(H_{0,\perp}))$ and
\begin{align}\label{def. nu(r)}
\nu(r):=\begin{cases}
\frac{d}{2r}-1,\quad &\frac{d}{2}\leq r\leq \frac{d+1}{2},\\
-\frac{1}{2r}\quad & \frac{d+1}{2}\leq r\leq\infty.
\end{cases}
\end{align}
Moreover, the estimate is sharp in the following sense: For every $k_0\geq K$ there exists $V\in L^r(\R^d)$, real-valued and $V\leq 0$, such that
\begin{align}\label{shrinking spectral cluster estimate sharpness}
\sigma(H)\cap\set{z\in\R}{|z-\lambda_{k_0}|\approx \|V\|_{L^r(\R^d)}\lambda_{k_0}^{\nu(r)}}\neq \emptyset.
\end{align} 
\end{theorem}




The proof of Theorem \ref{thm spectral clusters} will be based on the following proposition (here $R_{0,\perp}$ denotes the resolvent of  $H_{0,\perp}$). 

\begin{proposition}\label{proposition dispersive resolvent bound} Let $d\in 2\N$ and $q\in[2,2d/(d-2)]$. Assume $z\in\Lambda_{k_0}\cap\rho(H_{0,\perp})$. Then
\begin{align}\label{shrinking resolvent norm L-P}
\|R_{0,\perp}(z)\|_{L^{q'}(\R^{d})\to L^q(\R^{d})}\lesssim [1+\delta(z)^{-1}]\lambda_{k_0}^{\rho(q)},
\end{align}
where
\begin{align}\label{def. rho(q)}
\rho(q):=\begin{cases}
\frac{1}{q}-\frac{1}{2}\quad &2\leq q\leq \frac{2(d+1)}{d-1},\\
\frac{d-2}{2}-\frac{d}{q}\quad &\frac{2(d+1)}{d-1}\leq q\leq \infty.
\end{cases}
\end{align}
\end{proposition}

\begin{proof}
1. We first prove the bound
\begin{align}\label{L2 Lq' bound}
\|R_{0,\perp}(z)\|_{L^{q'}(\R^{d})\to L^2(\R^{d})}\lesssim [1+\delta(z)^{-1}]\lambda_{k_0}^{\rho(q)/2}.
\end{align}
Let $P_k$ be the spectral projection onto the eigenspace corresponding to $\lambda_k$. The dual version of the spectral projection estimates in \cite{MR2314091} reads
\begin{align}\label{Koch Ricci}
\|P_k\|_{L^{q'}(\R^{d})\to L^2(\R^d)}\lesssim \lambda_k^{\rho(q)/2},
\end{align}
By orthogonality of the spectral projections, we have
\begin{align*}
\|R_{0,\perp}(z)f\|_{L^{2}(\R^d)}^2&= \sum_{k=0}^{\infty}\frac{1}{|z-\lambda_k|^2}\|P_kf\|_{L^2(\R^d)}^2
\lesssim \|f\|_{L^{q'}(\R^d)}^2\sum_{k=0}^{\infty}\frac{\lambda_k^{\rho(q)}}{|z-\lambda_k|^2}\\
&\lesssim  \lambda_{k_0}^{\rho(q)}\left(1+\delta(z)^{-1}\right)^2\|f\|_{L^{q'}(\R^d)}^2.
\end{align*}
Here we used the following estimate in the last step,
\begin{align*}
&\sum_{k=0}^{\infty}\frac{\lambda_k^{\rho(q)}}{|z-\lambda_k|^2}
\lesssim \sum_{\lambda_k\leq\lambda_{k_0}/2} \frac{\lambda_{k}^{\rho(q)}}{\lambda_{k_0}^2}
+\sum_{\lambda_{k_0}/2<\lambda_k}\frac{\lambda_{k_0}^{\rho(q)}}{|\lambda_k-z|^2}
\lesssim  \lambda_{k_0}^{\rho(q)}\left(1+\delta(z)^{-1}\right)^2.
\end{align*}
This proves \eqref{L2 Lq' bound}.

2. Let $u\in C_0^{\infty}(\R^d)$ and $x_0\in \R^d$. Set $f:=(H_{0,\perp}-z)u$ and $\mu:=\re z$. Without loss of generality we may assume that $\mu>1$. We claim that
\begin{align}\label{dispersive bound localized to balls of radius lambda}
\mu^{\frac{1}{2(d+1)}}
\|u\|_{L^{\frac{2(d+1)}{d-1}}(B_{\sqrt{\mu}}(x_0))}
\lesssim \|u\|_{L^{2}(B_{2\sqrt{\mu}}(x_0))}
+\mu^{-\frac{1}{2(d+1)}}\|f\|_{L^{\frac{2(d+1)}{d+3}}(B_{2\sqrt{\mu}}(x_0))}
\end{align}
where the implicit constant is independent of $u,\mu$ and $x_0$. Proceeding as in \cite{MR2314091} we set
\begin{align}\label{new variables}
\overline{x}:=\frac{x-x_0}{\sqrt{\mu}},\quad \overline{y}:=\frac{y-y_0}{\sqrt{\mu}}
\end{align}
and
\begin{align}\label{new functions}
\overline{u}(\overline{x},\overline{y})=\e^{-\I(x_0y-y_0x)}u(x,y),\quad \overline{f}(\overline{x},\overline{y})=\e^{-\I(x_0y-y_0x)}f(x,y).
\end{align}
We then have 
\begin{align}\label{equation for Lmu}
L_{\mu}\overline{u}=\mu\overline{f},
\end{align}
where
\begin{align*}
L_{\mu}:=\sum_{j=1}^n\left(-\I\frac{\partial}{\partial x_j}+\frac{\mu}{2}\,y_j\right)^2+\left(-\I\frac{\partial}{\partial y_j}-\frac{\mu}{2}\,x_j\right)^2-\mu^2-\I(\im z)\mu.
\end{align*}
The operator $L_{\mu}$ is obviously normal and satisfies the assumptions of \cite[Theorem 7]{MR2252331} with $\delta\approx\mu^{-1}$ there. From \eqref{equation for Lmu} and  \cite[Theorem 7 B)]{MR2252331} it then follows that
\begin{align}\label{dispersive bound localized to balls of radius 1 Sobolev spaces}
\|\overline{u}\|_{W_{\mu}^{\frac{1}{d+1},\frac{2(d+1)}{d-1}}(B_{1}(0))}
\lesssim \mu^{1/2}\|\overline{u}\|_{L^{2}(B_{2}(0)}
+\mu\|\overline{f}\|_{W_{\mu}^{-\frac{1}{d+1},\frac{2(d+1)}{d+3}}(B_2(0))}.
\end{align}
Here, 
\begin{align*}
W_{\mu}^{s,p}:=\set{u\in \mathcal{S}'(\R^d)}{(\mu^2-\Delta)^{s/2}u\in L^p(\R^d)},\quad \|u\|_{W_{\mu}^{s,p}}:=\|(\mu^2-\Delta)^{s/2}u\|_{L^p}.
\end{align*}
Note that in the region $\{|\xi|\lesssim \mu\}$, we have $\|u\|_{W_{\mu}^{s,p}}\approx \mu^s\|u\|_{L^p}$, while in the (elliptic) region $\{|\xi|\gg \mu\}$, we have $\|u\|_{W_{\mu}^{s,p}}\gtrsim \mu^s\|u\|_{L^p}$ for $s\geq 0$ and $\|u\|_{W_{\mu}^{s,p}}\lesssim \mu^s\|u\|_{L^p}$ for $s\leq 0$. These estimates follow from standard Bernstein inequalities, see e.g.\ \cite[Appendix A]{MR2233925}. Therefore, \eqref{dispersive bound localized to balls of radius 1 Sobolev spaces} implies that
\begin{align*}
\mu^{\frac{1}{d+1}}\|\overline{u}\|_{L^{\frac{2(d+1)}{d-1}}(B_{1}(0))}
\lesssim \mu^{1/2}\|\overline{u}\|_{L^{2}(B_{2}(0)}
+\mu\mu^{-\frac{1}{d+1}}\|\overline{f}\|_{L^{\frac{2(d+1)}{d+3}}(B_2(0))}.
\end{align*}
By the change of variables \eqref{new variables}--\eqref{new functions}, this is equivalent to \eqref{dispersive bound localized to balls of radius lambda}.

3. Summing \eqref{dispersive bound localized to balls of radius lambda} over a finitely overlapping partition of $\R^d$ into balls of radius $\sqrt{\mu}$ yields
\begin{align}\label{dispersive bound}
\mu^{\frac{1}{2(d+1)}}
\|u\|_{L^{\frac{2(d+1)}{d-1}}(\R^d)}
\lesssim \|u\|_{L^{2}(\R^d)}
+\mu^{-\frac{1}{2(d+1)}}\|f\|_{L^{\frac{2(d+1)}{d+3}}(\R^d)}
\end{align}
Recalling that $\mu=\lambda_{k_0}+\mathcal{O}(1)$, $f=(H_{0,\perp}-z)u$, $\rho(q)=1/(d+1)$ for $q=2(d+1)/(d-1)$, and combining \eqref{L2 Lq' bound} with \eqref{dispersive bound}, we arrive at
\begin{align*}
\|u\|_{L^{\frac{2(d+1)}{d-1}}(\R^d)}
\lesssim (1+\delta(z)^{-1})\lambda_{k_0}^{-\frac{1}{d+1}}\|(H_{0,\perp}-z)u\|_{L^{\frac{2(d+1)}{d+3}}(\R^d)}.
\end{align*}
Since $C_0^{\infty}(\R^n)$ is a core for $H_{0,\perp}$ (\cite{MR1243098}), the above inequality is equivalent to \eqref{shrinking resolvent norm L-P} with $q=2(d+1)/(d-1)$; see e.g.\ the proof of \cite[Theorem C.3]{2015arXiv151000066C} for details of this argument. The general case follows by interpolation between this case and the cases $q=2$ and $q=2d/(d-2)$. In the former case, \eqref{shrinking resolvent norm L-P} is trivial. In the latter case, it follows from \cite[Theorem C.3]{2015arXiv151000066C}).
\end{proof}

\begin{proof}[Proof of Theorem \ref{thm spectral clusters}]
1. Let $q=2r'$. We then have $\nu(r)=\rho(q)$ (see \eqref{def. nu(r)} and \eqref{def. rho(q)}), and $\nu(r)\leq 0$ for $r\in [d/2,\infty]$. Moreover, in this range of $r$, we have $q=2r'\in[2,2d/(d-2)]$. Assume $z\in \sigma(H)\cap\Lambda_{k_0}$. By the Birman-Schwinger principle, H\"older's inequality and \eqref{shrinking resolvent norm L-P} we have
\begin{align}\label{Birman-Schwinger}
1\leq \|V^{1/2}R_{0,\perp}(z)|V|^{1/2}\|_{L^2(\R^d)\to L^2(\R^d)}\leq C 
\|V\|_{L^r(\R^d)}[1+\delta(z)^{-1}]\lambda_{k_0}^{\rho(q)}.
\end{align}
Here we have set $V^{1/2}:=|V|^{1/2}\sgn(V)$ with 
\begin{align*}
\sgn(z):=
\begin{cases}
\frac{\overline{z}}{|z|}&\quad \mbox{if }z\neq 0,\\
0&\quad \mbox{if }z=0.
\end{cases}
\end{align*}
Choosing $K$ so large that
\begin{align*}
C\|V\|_{L^r(\R^d)}\lambda_{K}^{\rho(q)}\leq \frac{1}{2},
\end{align*}
we infer from \eqref{Birman-Schwinger} that
\begin{align*}
\frac{1}{2}\leq C 
\|V\|_{L^r(\R^d)}\delta(z)^{-1}\lambda^{\rho(q)}.
\end{align*}
This proves \eqref{shrinking spectral cluster estimate}.

2. In order to show that the result is sharp we claim that it is sufficient to prove that for any fixed $k_0\in\N$ there exists $V\in L^r(\R^d)$ (depending on $k_0$) such that $\|V\|_{L^r}=1$, $V\leq 0$, and 
\begin{align}\label{claim sharpness}
\||V|^{1/2}P_{k_0}|V|^{1/2}\|\geq c_0 \lambda_{k_0}^{\rho(q)}.
\end{align}
Here, $c_0$ is some $k_0$-independent constant. To prove the claim, 
we define the Birman-Schwinger operators
\begin{align*}
Q(z;V):=|V|^{1/2}R_{0,\perp}(z)|V|^{1/2}
\end{align*}
where $z\in\rho(H_{0,\perp})$ and $V\in L^r(\R^d)$ satisfies \eqref{claim sharpness}.
Since
\begin{align*}
Q'(z;V)=|V|^{1/2}R_{0,\perp}(z)^2|V|^{1/2}\geq 0,\quad z\in\rho(H_{0,\perp})\cap\R,
\end{align*}
and $\|Q(z;V)\|\lesssim \lambda_{k_0}^{\rho(q)}<1$ for $z\in[\lambda_{k_0}-1/2,\lambda_{k_0}+1/2]\setminus\{\lambda_{k_0}\}$ and $k_0$ sufficiently large, the claim will follow by a standard application of the Birman-Schwinger principle once we prove that there exists $V\in L^r(\R^d)$ such that the operator $Q(a;V)$, with $a:=\lambda_{k_0}-\frac{1}{2}c_0\lambda_{k_0}^{\rho(q)}$, has an eigenvalue $\mu\geq 1$. We write
\begin{align*}
Q(a;V)=Q_0(a;V)+Q_1(a;V),
\end{align*}
where
\begin{align*}
Q_0(z;V):=\frac{1}{\lambda_{k_0}-z}|V|^{1/2}P_{k_0}|V|^{1/2},\quad Q_1(z;V):=\sum_{k\neq k_0}\frac{1}{\lambda_k-z}|V|^{1/2}P_{k}|V|^{1/2}.
\end{align*}
By \eqref{shrinking resolvent norm L-P}, we have $Q_1(a;V)=\mathcal{O}(\lambda_{k_0}^{\rho(q)})$. 
Moreover, since $Q_0(a;V)$ is nonnegative and compact \cite[Lemma 5.1]{MR1044429}, it follows that $\mu_0(V):=\|Q_0(a;V)\|$ is its largest eigenvalue. Let $\psi\in L^2(\R^d)$ be the corresponding normalized eigenfunction. Then
\begin{align*}
\|(Q(a;V)-\mu_0(V))\psi\|_2=\|Q_1(a;V)\psi\|_2=\mathcal{O}(\lambda_{k_0}^{\rho(q)}).
\end{align*}
Since $Q(a;V)$ is selfadjoint, this implies that
\begin{align}\label{nonempty intersection}
\sigma(Q(a;V))\cap [\mu_0(V)-\mathcal{O}(\lambda_{k_0}^{\rho(q)}),\mu_0(V)+\mathcal{O}(\lambda_{k_0}^{\rho(q)})]\neq \emptyset.
\end{align} 
Choosing $c=c_0/2$, we have by \eqref{claim sharpness}
\begin{align}\label{mu0}
\mu_0(V)=\|Q_0(a;V)\|=\frac{1}{c}\lambda_{k_0}^{-\rho(q)}\|V^{1/2}P_{k_0}V^{1/2}\|\geq \frac{c_0}{c}= 2.
\end{align}
It follows from \eqref{nonempty intersection} that $Q(a:V)$ has an eigenvalue $\mu\geq 1$ for $k_0$ sufficiently large. 

3. It remains to prove the claim \eqref{claim sharpness}. We use the fact that the spectral projection estimates \eqref{Koch Ricci} are sharp. In the $TT^*$ version, this means that
\begin{align}\label{sharpness spectral projection estimates TT*}
\|P_k\|_{L^{q'}\to L^q}\geq 2c_0\lambda_{k}^{\rho(q)},\quad k\in\N.
\end{align}
By H\"older's inequality and a duality argument,
we have
\begin{align}\label{Holder duality}
\|P_k\|_{L^{q'}\to L^q}=\sup_{\|W_1\|_{L^{2r}}=\|W_2\|_{L^{2r}}=1}\|W_1P_kW_2\|.
\end{align} 
Moreover, the Cauchy-Schwarz inequality yields
\begin{align*}
|\langle W_1P_kW_2f,g\rangle|&=|\langle P_kW_2f,P_kW_1g\rangle|\leq \|P_kW_2f\|\|P_kW_1g\|\\
&=\langle W_2P_kW_2f,f\rangle^{1/2}\langle W_1P_kW_1g,g\rangle^{1/2}\\
&\leq \|W_1P_kW_1\|^{1/2}\|W_2P_kW_2\|^{1/2}\|f\|\|g\|,
\end{align*}
and hence
\begin{align}\label{Cauchy-Schwarz}
\|W_1P_kW_2\|\leq\|W_1P_kW_1\|^{1/2}\|W_2P_kW_2\|^{1/2}.
\end{align}
Combining \eqref{sharpness spectral projection estimates TT*}--\eqref{Cauchy-Schwarz}, we get
\begin{align*}
\sup_{\|W\|_{L^{2r}}=1}\|WP_kW\|\geq 2c_0\lambda_{k}^{\rho(q)}.
\end{align*}
Therefore, we can choose a normalized $W\in L^{2r}(\R^d)$ such that 
\begin{align*}
\|WP_{k_0}W\|\geq c_0\lambda_{k_0}^{\rho(q)}. 
\end{align*}
The claim \eqref{claim sharpness} follows with $V=W^2$.
\end{proof}

\section{Limiting absorption principle and unique continuation in odd dimensions}

\subsection{Limiting absorption principle}

Consider the Hamiltonian with constant magnetic field in $d=2n+1$ dimensions.
\begin{align}\label{eq. unperturbed Landau Hamiltonian odd d}
H_0=\sum_{j=1}^n\left(-\I\frac{\partial}{\partial_{x_j}}+\frac{y_j}{2}\right)^2+\left(-\I\frac{\partial}{\partial_{y_j}}-\frac{x_j}{2}\right)^2-\frac{\partial^2}{\partial z^2},\quad (x,y,z)\in\R^{2n+1}.
\end{align}
The spectrum of $H_0$ is purely absolutely continuous and $\{\lambda_k\}_{k\in\N}$ play the role of thresholds.
%
%
Writing $x_{\perp}=(x,y)\in\R^{2n}$ and $z=x_d\in\R$ we introduce the following mixed Lebesgue spaces:
\begin{align*}
\mathcal{X}_q&:=\left(L_{x_d}^{\frac{2}{1-2\rho(q)}}(\R)\cap L_{x_d}^{1}(\R)\right)L_{x^{\perp}}^{q'}(\R^{2n}),\\
\mathcal{V}_q&:=\left(L_{x_d}^{-\frac{1}{2\rho(q)}}(\R)\cap  L_{x_d}^{1}(\R)\right)L_{x^{\perp}}^{r}(\R^{2n}),\\
\mathcal{X}_q^*&:=\left(L_{x_d}^{\frac{2}{1+2\rho(q)}}(\R)+L_{x_d}^{\infty}(\R)\right)L_{x^{\perp}}^{q}(\R^{2n}).
\end{align*}
Here $q$ and $r$ are related by $q=2r'$. As before, we will sometimes omit $\R$ and $\R^{2n}$ from the notation of Lebesgue spaces. The resolvents of $H_0$ and $H_0+V$ will be denoted by $R_0$ and $R$, respectively.

Our main result in this subsection is the following limiting absorption principle.

\begin{theorem}[Limiting Absorption Principle]\label{thm LAP}
Let $d\geq 3$, $d\in 2\N+1$, $r\in (d/2,\infty)$ and $q=2r'\in (2,2d/(d-2))$. Let $J$ be a compact subset of $\R\setminus\{\lambda_k\}_{k\in\N}$. If $\|V\|_{\mathcal{V}_q}$ is sufficiently small, then
\begin{align}\label{LAP small perturbation}
\sup_{\stackrel{\lambda\in J}{\epsilon\in[-1,1]\setminus\{0\}}}\|R(\lambda+\I\epsilon)\|_{\mathcal{X}_q\to\mathcal{X}_q^*}<\infty.
\end{align}
\end{theorem}

\begin{lemma}\label{lemma resolvent fixed xd}
Assume $z\in\Lambda_{k_0}$, $k_0\in\N$, and $0<|\im z|<1$. Then for $2\leq q\leq \infty$, we have the estimate
\begin{align*}
\|R_0(z)f(\cdot,x_d)\|_{L^q(\R^{2n})}&\lesssim \int_{-\infty}^{\infty}\left\{|x_d-y_d|^{-1-2\rho(q)}\right.\\
&\left.+k_0^{\rho(q)}(k_0^{\frac{1}{2}}+\delta(z)^{-\frac{1}{2}})\right\}\|f(\cdot,y_d)\|_{L^{q'}(\R^{2n})}\rd y_d
\end{align*}
for every $f\in C_c^{\infty}(\R^d)$. Here, $\Lambda_{k_0}$, $\delta(z)$ and $\rho(q)$ are defined in \eqref{def. Lambdak0} and\eqref{def. rho(q)}.
\end{lemma}

\begin{proof}
We have
\begin{align*}
R_0(z)f=\sum_{k=0}^{\infty}(P_k\otimes (-\partial_{x_d}^2-(z-\lambda_k))^{-1})f.
\end{align*}
The resolvent kernel of $(-\partial_{x_d}^2-\mu)^{-1}$ is given by
\begin{align*}
(-\partial_z^2-\mu)^{-1}(x_d,y_d)=\frac{\e^{\I\sqrt{\mu}|x_d-y_d|}}{2\I\sqrt{\mu}}
\end{align*}
where $\sqrt{\cdot}:\C\setminus[0,\infty)\to\C^+$ is the principal branch of the square root. Therefore,
\begin{align*}
R_0(z)f(\cdot,x_d)=\sum_{k=0}^{\infty}\int_{-\infty}^{\infty}\frac{\e^{\I\sqrt{z-\lambda_k}|x_d-y_d|}}{2\I\sqrt{z-\lambda_k}}(P_k\otimes \mathbf{1})f(\cdot,y_d)\rd y_d.
\end{align*}
By Minkowski's inequality, it follows that
\begin{align*}
\|R_0(z)f(\cdot,x_d)\|_{L^q(\R^{2n})}\leq \sum_{k=0}^{\infty}\int_{-\infty}^{\infty}\frac{\e^{-\im\sqrt{z-\lambda_k}|x_d-y_d|}}{|z-\lambda_k|^{1/2}}\|(P_k\otimes \mathbf{1})f(\cdot,y_d)\|_{L^q(\R^{2n})}\rd y_d.
\end{align*}
Thus, by \cite{MR2314091}, 
\begin{align*}
\|R_0(z)f(\cdot,x_d)\|_{L^q(\R^{2n})}\lesssim \sum_{k=0}^{\infty}\lambda_{k}^{\rho(q)}\int_{-\infty}^{\infty}\frac{\e^{-\im\sqrt{z-\lambda_k}|x_d-y_d|}}{|z-\lambda_k|^{1/2}}\|f(\cdot,y_d)\|_{L^{q'}(\R^{2n})}\rd y_d.
\end{align*}
By Fubini's theorem it remains to prove that
\begin{align}\label{Claim bound for sum over k}
 \sum_{k=0}^{\infty}\lambda_{k}^{\rho(q)}\frac{\e^{-\im\sqrt{z-\lambda_k}|t|}}{|z-\lambda_k|^{1/2}}
 \lesssim |t|^{-1-2\rho(q)}+k_0^{\rho(q)}(k_0^{\frac{1}{2}}+\delta(z)^{-\frac{1}{2}}).
\end{align}
We write $z=\lambda+\beta+\I\tau$, where $|\beta|,|\tau|\leq 1$, and $w=z-\lambda_k$. Then for $|k-k_0|\geq 1$ we have
\begin{align*}
|w|^2=(2(k_0-k)+\beta)^2+\tau^2\geq 2|k-k_0|-|\beta|\geq |k-k_0|.
\end{align*}
Moreover, if $k>2k_0$, then $\re w<0$, which implies that 
\begin{align*}
\im\sqrt{w}\geq \frac{\sqrt{|w|}}{\sqrt{2}}\geq \frac{\sqrt{k}}{2},\quad k>2k_0.
\end{align*}
We can thus estimate the sum in \eqref{Claim bound for sum over k} by
\begin{align*}
\sum_{k=1}^{k_0-2}\frac{k^{\rho(q)}}{(k_0-k)^{1/2}}
+\sum_{k=k_0-1}^{k_0+1}\frac{k^{\rho(q)}}{\delta(z)^{1/2}}
+\sum_{k=k_0+2}^{2k_0}\frac{k^{\rho(q)}}{(k-k_0)^{1/2}}\\
+\sum_{k>2k_0}\frac{k^{\rho(q)}\e^{-\frac{1}{2}\sqrt{k}|t|}}{\sqrt{k/2}}+\mathcal{O}(1).
\end{align*}
Consider the functions
\begin{align*}
(0,k_0-1)&\ni r\mapsto \frac{r^{\rho(q)}}{(k_0-r)^{1/2}},\\
(k_0+1,2k_0)&\ni r\mapsto \frac{r^{\rho(q)}}{(r-k_0)^{1/2}},\\
(2k_0,\infty)&\ni r\mapsto \frac{r^{\rho(q)}\e^{-\frac{1}{2}\sqrt{r}|t|}}{\sqrt{r}}.
\end{align*}
The last two are monotonically decreasing, while the first has a single change of monotonicity at
$r=k_0^{-1}(1/(2\rho(q))+1)^{-1}$. Hence \eqref{Claim bound for sum over k} is bounded by
\begin{align*}
\int_0^{k_0-1}\frac{r^{\rho(q)}}{(k_0-r)^{1/2}}\rd r+3\frac{k_0^{\rho(q)}}{\delta(z)^{1/2}}
+\int_{k_0+1}^{2k_0}\frac{r^{\rho(q)}}{(r-k_0)^{1/2}}\rd r\\
+\int_{2k_0}^{\infty}\frac{r^{\rho(q)}\e^{-\frac{1}{2}\sqrt{r}|t|}}{\sqrt{r}}\rd r+\mathcal{O}(1).
\end{align*}
Note that the errors made by replacing sums by integrals have been absorbed in the $\mathcal{O}(1)$ term.
Splitting the first integral into a contribution from $\{r>k_0/2\}$ and its complement and changing variables $r\to t^2r$ in the last integral, we finally obtain the estimate \eqref{Claim bound for sum over k}.
\end{proof}

\begin{proof}[Proof of Theorem \ref{thm LAP}]
1. Assume first that $V=0$. By duality and by density of $C_0^{\infty}(\R^d)$ in $\mathcal{X}_q$, it suffices to prove that
\begin{align}\label{duality XqXq*}
\sup_{\stackrel{f,g\in  C_0^{\infty}(\R^d)}{\|f\|_{\mathcal{X}_q}=\|g\|_{\mathcal{X}_q}=1}}|\langle R_0(\lambda+\I\epsilon)f,g\rangle|\leq C
\end{align}
for all $\lambda\in J$ and $\epsilon\neq 0$, with a constant $C$ independent of $\lambda$ and $\epsilon$. 
Hence, let $f,g\in C_0^{\infty}(\R^d)$.
By H\"older's inequality and Lemma \ref{lemma resolvent fixed xd}, we have the estimate 
\begin{align*}
&|\langle R_0(\lambda+\I\epsilon)f,g\rangle|
=\left|\int_{-\infty}^{\infty}\langle R_0(\lambda+\I\epsilon)f(\cdot,x_d),g(\cdot,x_d)\rangle_{L^2_{x^{\perp}}}\rd x_d\right|\\
&\leq \int_{-\infty}^{\infty}
\| R_0(\lambda+\I\epsilon)f(\cdot,x_d)\|_{L^{q}_{x^{\perp}}}
\|g(\cdot,x_d)\|_{L^{q'}_{x^{\perp}}}\rd x_d\\
&\lesssim_J \int_{-\infty}^{\infty}\int_{-\infty}^{\infty}\left\{|x_d-y_d|^{-1-2\rho(q)}+1
\right\}\|f(\cdot,y_d)\|_{L^{q'}_{x^{\perp}}}\|g(\cdot,x_d)\|_{L^{q'}_{x^{\perp}}}\rd y_d\rd x_d.
\end{align*}
By the one-dimensional Hardy-Littlewood-Sobolev inequality,
\begin{align*}
\int_{-\infty}^{\infty}\int_{-\infty}^{\infty}|x_d-y_d|^{-1-2\rho(q)}\|f(\cdot,y_d)\|_{L^{q'}_{x^{\perp}}}\|g(\cdot,x_d)\|_{L^{q'}_{x^{\perp}}}\rd y_d\rd x_d\lesssim 
\|f\|_{L^p_{x_d}L^{q'}_{x^{\perp}}}
\|g\|_{L^p_{x_d}L^{q'}_{x^{\perp}}},
\end{align*}
where $p=\frac{2}{1-2\rho(q)}$,
and we observe that\footnote{The minimum of $\rho(q)$ is attained at $q=\frac{2(d+1)}{d-1}$.}
\begin{align*}
q\in (2,2d/(d-2))\implies -\frac{1}{d+1}\leq \rho(q)< 0\implies -2\rho(q)\in (0,1).
\end{align*}
To finish the proof, we observe the trivial identity
\begin{align*}
\int_{-\infty}^{\infty}\int_{-\infty}^{\infty}\|f(\cdot,y_d)\|_{L^{q'}_{x_{\perp}}}\|g(\cdot,x_d)\|_{L^{q'}_{x_{\perp}}}\rd y_d\rd x_d= \|f\|_{L^1_{x_d}L^{q'}_{x_{\perp}}}
\|g\|_{L^1_{x_d}L^{q'}_{x_{\perp}}}.
\end{align*}
We conclude that
\begin{align*}
|\langle R_0(z)f,g\rangle|\lesssim_{J} \|f\|_{\mathcal{X}_q}
\|g\|_{\mathcal{X}_q},
\end{align*}
and \eqref{duality XqXq*} follows.

2. Assume now $V\neq 0$. By H\"older's inequality, we have
\begin{align*}
\||V|^{1/2}R_0(\lambda+\I\epsilon)V^{1/2}\|\leq \|V\|_{\mathcal{V}_q}\|R_0(\lambda+\I\epsilon)\|_{\mathcal{X}_q\to\mathcal{X}_q^*}.
\end{align*}
Hence, by the first part of the proof, 
\begin{align*}
\||V|^{1/2}R_0(\lambda+\I\epsilon)V^{1/2}\|<1
\end{align*}
whenever $\|V\|_{\mathcal{V}_q}$ is sufficiently small. Using the resolvent identity
\begin{align*}
&W_1R(z)W_2=W_1R_0(\lambda+\I\epsilon)W_2\\
&-W_1R_0(\lambda+\I\epsilon)V^{1/2}(1-|V|^{1/2}R_0(\lambda+\I\epsilon)V^{1/2})^{-1}|V|^{1/2}R_0(\lambda+\I\epsilon)W_2,
\end{align*}
and a geometric series argument, we conclude that
\begin{align*}
\sup_{\stackrel{\lambda\in J}{\epsilon\neq 0}}\|W_1R(\lambda+\I\epsilon)W_2\|\lesssim_J \frac{1+\|V\|_{\mathcal{V}_q}}{1-\|V\|_{\mathcal{V}_q}}\|W_1\|_{\mathcal{V}_q}\|W_2\|_{\mathcal{V}_q}.
\end{align*}
By duality, this implies \eqref{LAP small perturbation}.
\end{proof}

\subsection{Unique continuation}

Recall the definition of the weak unique continuation property (w.u.c.p.): A partial differential operator $P(x,D)$ is said to have the w.u.c.p. if the following holds. Let $\Omega\subset\R^d$ be open and connected, and assume that $P(x,D)u=0$ in $\Omega$ where $u$ is compactly supported in $\Omega$. Then $u\equiv 0$ in $\Omega$. 

\begin{theorem}\label{thm unique continuation}
Let $d\geq 3$, $d\in 2\N+1$. Assume that $V\in L^{d/2}(\R^d)$. Then $H_0+V$ has the w.u.c.p.
\end{theorem}

The proof is a standard application (see e.g. \cite{MR1013816}) of the following new Carleman estimate.

\begin{theorem}\label{thm Carleman}
Let $d\geq 3$, $d\in 2\N+1$, and let $I\subset \R$ be a compact interval. There exists a constant $C_I>0$ such that for any $u\in C_c^{\infty}(\R^{2n}\times I)$ and $\tau\in\R$, with $\dist(\tau^2,2\N+n)\geq 1/2$, we have the estimate
\begin{align}\label{eq Carleman estimate}
\|\e^{\tau x_d}u\|_{L^{\frac{2d}{d-2}}(\R^d)}\leq C_I \|\e^{\tau x_d}H_0u\|_{L^{\frac{2d}{d+2}}(\R^d)}.
\end{align}
\end{theorem}

\begin{proof}[Proof of Theorem \ref{thm Carleman}]
We follow the procedure of Jerison's \cite{MR865834} proof of the unique continuation theorem of Jerison and Kenig \cite{MR794370}. The proof is similar to \cite[Theorem 1.2]{MR3005546}, except that we use the spectral projection estimates of Koch and Ricci \cite{MR2314091} for the twisted Laplacian \eqref{eq. unperturbed Landau Hamiltonian even d} instead of the spectral cluster estimates of Sogge \cite{Sogge1988}. We recall (a special case of) the main result in \cite{MR2314091}:
\begin{align}\label{Koch Ricci for Carleman}
\|P_k u\|_{L^{\frac{2d}{d-2}}(\R^{2n})}&\lesssim \lambda_k^{-\frac{1}{d}}\|u\|_{L^2(\R^{2n})},\\
\|P_k u\|_{L^2(\R^{2n})}&\lesssim \lambda_k^{-\frac{1}{d}}\|u\|_{L^{\frac{2d}{d+2}}(\R^{2n})}.
\end{align}
Adopting the notation of \cite{MR3005546}, we denote by $G_{\tau}$ the inverse of the conjugated operator  
\begin{align*}
\e^{\tau x_d}H_0\e^{-\tau x_d}=D_{x_d}^2+2\I\tau D_{x_d}-\tau^2+H_{0,\perp}.
\end{align*}
Using the eigenfunction expansion of $H_{0,\perp}$, we obtain 
\begin{align*}
G_{\tau}f(x_{\perp},x_d)=\sum_{k=0}^{\infty}\int_{-\infty}^{\infty}m_{\tau}(x_d-y_d,\lambda_k)P_kf(x_{\perp},y_d)\rd y_d
\end{align*}
where 
\begin{align*}
m_{\tau}(x_d-y_d,\lambda_k)=\frac{1}{2\pi}\int_{-\infty}^{\infty}\frac{\e^{\I(x_d-y_d)\eta}}{\eta^2+2\I\tau\eta-\tau^2+\lambda_k^2}\rd \eta.
\end{align*}
Using the spectral projection estimates \eqref{Koch Ricci for Carleman}
and proceeding as in the proof of \cite{MR3005546}, we arrive at
\begin{align*}
\|G_{\tau}f(\cdot,x_d)\|_{L^{\frac{2d}{d-2}}(\R^{2n})}\lesssim 
\sum_{k=0}^{\infty}(1+2k)^{-\frac{1}{d}}\int_{-\infty}^{\infty}|m_{\tau}(x_d-y_d,\lambda_k)|\|f(\cdot,y_d)\|_{L^{\frac{2d}{d+2}}(\R^{2n})}\rd y_d.
\end{align*}
By the straightforward estimate 
\begin{align*}
|m_{\tau}(x_d-y_d,\lambda_k)|\leq \frac{\e^{-|\tau-\lambda_k||x_d-y_d|}}{\lambda_k},
\end{align*}
see Lemma 2.3 in \cite{MR3005546}, one can sum up the previous estimates (estimate the sum by an integral and change variables $k\to \lambda=\sqrt{2k+n}$):
\begin{align*}
\sum_{k=0}^{\infty}(1+2k)^{-\frac{1}{d}}|m_{\tau}(x_d-y_d,\lambda_k)|\lesssim \int_0^{\infty}
\lambda^{-\frac{2}{d}}\e^{-|\tau-\lambda||x_d-y_d|}\rd\lambda\lesssim 1+|x_d-y_d|^{\frac{2}{d}-1}.
\end{align*}
Thus
\begin{align*}
\|G_{\tau}f(\cdot,x_d)\|_{L^{\frac{2d}{d-2}}(\R^{2n})}
&\lesssim|I|^{\frac{1}{2}-\frac{1}{d}}\|f\|_{L^{\frac{2d}{d+2}}(\R^{d})}\\
&+\int_{-\infty}^{\infty}|x_d-y_d|^{\frac{2}{d}-1}\|f(\cdot,y_d)\|_{L^{\frac{2d}{d+2}}(\R^{2n})}\rd y_d.
\end{align*}
An application of the one-dimensional Hardy-Littlewood-Sobolev inequality yields
\begin{align*}
\|G_{\tau}f\|_{L^{\frac{2d}{d-2}}(\R^{d})}\lesssim \|f\|_{L^{\frac{2d}{d+2}}(\R^{d})}.
\end{align*}
This completes the proof.
\end{proof}


\bibliographystyle{plain}
\bibliography{C:/Users/Jean-Claude/Desktop/papers/bibliography_masterfile}


\end{document}